\documentclass[12pt, a4paper]{amsart} 
\usepackage[T2A]{fontenc}
\usepackage[utf8]{inputenc}
\usepackage{geometry}
\usepackage{amsmath, amssymb, amsfonts, amsthm}
\usepackage{enumitem}
\usepackage{graphicx}
\usepackage{tikz}
\usepackage{esvect}
\usepackage{hyperref}
\usepackage{bbm}

\usepackage{comment}
\usepackage[utf8]{inputenc}
\usepackage[T2A]{fontenc}
\usepackage[english]{babel}
\usepackage{xcolor}
\usepackage{wrapfig}

\newtheorem{theorem}{Theorem}
\newtheorem{theom}{Theorem}

\newtheorem{lam}{Lemma}

\newtheorem{lemma}{Lemma}
\newtheorem{prop}{Proposition}
\newtheorem{defin}{Definition}
\theoremstyle{remark}
\newtheorem*{rem}{Remark}
\newtheorem*{cor}{Corollary}
\newenvironment{proof}{\noindent{\bf Proof:}}{$\hfill \Box$ \vspace{10pt}}  
\def\Xint#1{\mathchoice
{\XXint\displaystyle\textstyle{#1}}%
{\XXint\textstyle\scriptstyle{#1}}%
{\XXint\scriptstyle\scriptscriptstyle{#1}}%
{\XXint\scriptscriptstyle\scriptscriptstyle{#1}}%
\!\int}
\def\XXint#1#2#3{{\setbox0=\hbox{$#1{#2#3}{\int}$ }
\vcenter{\hbox{$#2#3$ }}\kern-.6\wd0}}

\def\dashint{\Xint-}
\newcommand{\hel} {
\hskip2.5pt{\vrule height7pt width.5pt depth0pt}
\hskip-.2pt\vbox{\hrule height.5pt width7pt depth0pt}
\, }
\newcommand{\restr}{\hel}

\begin{document}

%\title[A new uncertainty principle]{A new uncertainty principle : small non-zero functions with small support.}
\title{A generalization of the Beurling--Malliavin Majorant Theorem}
\author{Ioann Vasilyev}
%\thanks{The author was supported by the Russian Science Foundation (grant No.~18-11-00053).}
%\address{Universit\'e Paris-Saclay, LAMA (UMR 8050), 61 avenue du G\'en\'eral de Gaulle, 94010, Cr\'eteil, France}
\address{Universit\'e Paris-Saclay, CNRS, Laboratoire de math\`ematiques d'Orsay, 91405, Orsay, France}
%\address{St.-Petersburg Department of V.A. Steklov Mathematical Institute, Russian Academy of Sciences (PDMI RAS), Fontanka 27, St.-Petersburg, 191023, Russia}
%\email{milavas@mail.ru}
\email{ioann.vasilyev@u-psud.fr}

\begin{abstract}
In this article we prove a generalization of the Beurling--Malliavin Majorant Theorem.  In more detail, we establish a new sufficient condition for a function to be a Beurling--Malliavin majorant. Our result is strictly more general than that of the Beurling--Malliavin Majorant Theorem. We also show that our result is sharp in a number of senses.
\end{abstract}

%% Beginn des Briefes %%%%%%%%%%%%%%%%%%%%%%%%%%%%%%%%%%%%%%%
\maketitle
\section{Introduction}
Let $\mathrm{Lip}(\mathbb R)$ denote the space of Lipschitz functions in $\mathbb R$ (i.e. functions $f$ satisfying for all $x,y\in \mathbb R$ the following inequality: $|f(x)-f(y)|\leq C|x-y|$ with $C>0$ independent of $x,y$). By $\mathrm{Lip}(\xi, \mathbb R)$ we shall denote all Lipschitz functions in $\mathbb R$ with the Lipschitz constant $\xi$.

The following theorem was first proved by A. Beurling and P. Malliavin, see~\cite{bermal1}.
\begin{theom}
\label{thm-1}(Beurling--Malliavin)
Let $\omega:\mathbb R\rightarrow (0,1]$ be a function such that $\log(1/\omega)\in L^1(\mathbb R, dx/(1+x^2))$, and $\log(1/\omega)$ is a Lipschitz function. Then for each $\delta>0$ there exists a function $f\in L^2(\mathbb R)$, which is not identically zero and satisfying $\mathrm{spec}(f)\subset [0,\delta]$ and $|f(x)|\leq \omega(x)$ for all $x\in \mathbb R$. 
%\end{lemma}
\end{theom}

For a function $f\in L^2(\mathbb R)$, by $\mathrm{spec}(f)$ we mean the spectrum of $f$, i.e. the support of its Fourier transformation. Note that the spectrum is defined up to a set of the Lebesgue measure zero. Let us also remark that here, the term ``not identically zero'' means ``not zero almost everywhere''. We shall further sometimes write just ``nonzero'' for brevity.

The Beurling--Malliavin theorems are considered by many experts to be among the most deep and important results of the 20th century Harmonic Analysis. Theorem~\ref{thm-1} above is called the Beurling--Malliavin Majorant Theorem (or the First Beurling--Malliavin Theorem). This result gives conditions for the majorant $\omega$ ensuring existence of nonzero function whose spectrum lies in an arbitrary small interval and whose modulus is majorized by $\omega$. This theorem is a crucial tool in the proof of the Second Beurling--Malliavin Theorem about the radius of completeness of an exponential system. Moreover, Theorem~\ref{thm-1} was recently used by J. Bourgain and S. Dyatlov in the theory of resonances for hyperbolic surfaces, see~\cite{bourdyat}. Deep connections of the First Beurling--Malliavin Theorem with nowadays popular gap and type problems are discussed in papers~\cite{borsod},~\cite{polt} and~\cite{makpol}.

Note that Theorem~\ref{thm-1} is in a certain sense a contradiction to the following postulate, called the uncertainty principle: ``It is impossible for a nonzero function and its Fourier transform to be simultaneously very small, unless the function is zero''. Indeed, Theorem~\ref{thm-1} shows that there exist nonzero functions that are ``small'' and whose Fourier transforms are also ``small''. Of course these smallnesses are different from each other and from the smallness in $L^2(\mathbb R)$. So in fact Theorem~\ref{thm-1} does not contradict the most well known variant of uncertainty principle, the Heisenberg inequality. For recent violations of the uncertainty principle of a completely different nature, see papers~\cite{polkis} and~\cite{olev}.

In addition to the original proof of A. Beurling and P. Malliavin, there are many approaches to the proof of Beurling--Malliavin theorems due to H. Redheffer (see~\cite{redh}), L. De Branges (see~\cite{debr}), P. Kargaev (see~\cite{koos2.5}), N. Makarov and A. Poltoratskii (see~\cite{makpol}), to name just some of them. In 2005 V. Havin, J. Mashreghi and F. Nazarov~\cite{nazhav} suggested a new proof of the First Beurling--Malliavin Theorem. An essential novelty of their proof was that it was done by (almost) purely real methods and did not use complex analysis except at one place, see~\cite{nazhav} and the remark right after the formulation of Theorem~\ref{diakon} below. 

\bigskip
Among the goals of the present paper is to give a proof of a new nontrivial generalization of Theorem~\ref{thm-1}. Before stating our main results, we recall some classical definitions and fix some notations.

One of the principle objects of this paper is the class of BM majorants.

\begin{defin}
\label{horoshayamajoranta}
Let $\omega$ be a bounded nonnegative function on $\mathbb R$. This function is called a Beurling--Malliavin majorant (we shall further write ``BM majorant'' to save space), if for any $\sigma>0$ there exists a nonzero function $f\in L^2(\mathbb R)$ such that
$$(a)\, |f|\leq \omega ; \; (b)\, \mathrm{spec} (f) \subset [0,\sigma].$$
The set of all BM majorants will be further referred to as the BM class.
If the conditions $(a)$ and $(b)$ just above are satisfied for a function $\omega$ with some fixed $\sigma>0$, then we call such function $\omega$ a $\sigma$-admissible majorant. If  we replace the condition $(a)$ with a stronger two-sided condition $C\omega\leq|f|\leq \omega$ for some constant $C>0$, then what we get is the definition of a strictly admissible majorant.
\end{defin}

%Let $\beta$be a positive number. 
Recall that the \textit{Poisson measure} $dP$ on $\mathbb R$ is defined by the following formula 
$$dP(x):=\frac{dx}{1+x^2}.$$
%If $\beta=2$, we shall omit the subscript $2$ and simply write $dP$ instead of $dP_2$.
The corresponding weighted Lebesgue space $L^1(dP)$ is the space of all functions $f$ satisfying 
$\int_{\mathbb R}|f|dP<\infty$. The expression $\int_{\mathbb R}\log(1/\omega)dP$ will be sometimes further referred to as the logarithmic integral of $\omega.$

Note that the condition $\log(1/\omega)\in L^1(dP)$ is necessary for $\omega$ to be a $BM$ majorant, but not sufficient, see~\cite{nazhav}. What Theorem~\ref{thm-1} establishes is that some additional regularity suffices for admissibility.

We remind the reader of how one should modify the Cauchy kernel in order to extend the definition of the Hilbert transformation up to the space $L^1(dP).$

\begin{defin}
The \textit{Hilbert transformation} of a function $f\in L^{1}(dP)$ is defined as the following principal value integral
$$\mathcal H f(x):=\dashint_{\mathbb R}\Bigl(\frac{1}{x-t}+\frac{t}{t^2+1}\Bigr)f(t)dt.
$$
%Here, the normalization constant $c_d$ is given by $c_d:=(\pi \gamma_{d-1})^{-1}$, where $\gamma_{d-1}$ is the Euclidean volume of the $(d-1)$-dimensional Euclidean ball. 
It is worth noting that the integral above converges for almost all $x\in\mathbb R$.
\end{defin}

To avoid ambiguity, we stress that this definition coincides, up to an additive constant, with the classical one for functions in $L^1(\mathbb R)$.

Let us now introduce function classes that will play an important role in what follows. To this end, we first define an auxiliary system of intervals: $J_0=[-2,2),$ and for $j\in \mathbb N$, $$J_{j}=[2^{j},2^{j+1}),J_{-j}=[-2^{j+1},-2^{j}).$$

\begin{defin}
\label{opredelenie}
Let $\beta\in (0,1]$. %Consider the following system of intervals: $J_0:=[-2,2),$ and for $j\in \mathbb N, J_{j}:=[2^{j},2^{j+1})$ and $J_{-j}:=[-2^{j+1},-2^{j}).$ 
 If $\beta<1$, then we shall say that an absolutely continuous function $\varphi$ belongs to the class $V_{\beta}$ if $\varphi$ is a $\beta$-H\"older function on the interval $J_j$  with the constant $\kappa_j$ and moreover these constants satisfy
%$$
%\biggl(\sup_{j\in \mathbb Z} \frac{1}{|J_j|}\int_{J_j}|\varphi^{\prime}(x)|^r dx\biggr)^{1/r}
\begin{equation}
\label{ryad}
\left(\sum_{n\in \mathbb Z} 2^{-|j|}\kappa_j^{1/(1-\beta)}\right)^{1-\beta}<\infty.
\end{equation}
%The expression in the left hand part of the last line will be further denoted $\|\varphi\|_{V_{\beta}}.$
In the case when $\beta=1$, we use the following convention $V_{\beta}=\mathrm{Lip}(\mathbb R).$
\end{defin}
% this class is wider then L^r(\mathbb R,dx/(1+|x|))

Note that these classes resemble homogeneous weighted Sobolev spaces. We are going to work with functions that belong to intersections $L^1(dP)\bigcap V_\beta$. From the functional-analytic point of view, these intersections are Banach spaces with respect to the norms $\|\cdot\|_{L^1(dP)} +\|\cdot\|_{V_{\beta}}$.

We are now in position to formulate the first main result of this paper to be proved in the next section.
\begin{theorem}
\label{thm0}
Let $\omega:\mathbb R\rightarrow (0,1]$ be a function such that $\log(1/\omega)\in L^1(dP)$, with $\log(1/\omega)$ absolutely continuous and satisfying $\log(1/\omega)\in V_{\beta}$ for some $\beta\in(0,1]$. Then for each $\delta>0$ there exists a function $f\in L^2(\mathbb R)$, not identically zero, such that $\mathrm{spec}(f)\subset [0,\delta]$ and $|f(x)|\leq \omega(x)$ for all $x\in \mathbb R$. 
%\end{lemma}
\end{theorem}

\begin{rem}
We would like to stress that one can replace the intervals $J_j$ in the definition of the spaces $V_\beta$ with any system of intervals $[\lambda_j,\lambda_{j+1})$ where $\{\lambda_j\}$ is  any sequence of reals satisfying $\lambda\leq \lambda_{j+1}/\lambda_j \leq\Lambda$ with $1<\lambda<\Lambda<\infty$, in a way that the corresponding version of Theorem~\ref{thm0} holds true.
\end{rem}

\begin{rem}
Throughout this paper, $\Omega$ will mean $\log(1/w)$ for a function $\omega:\mathbb R\rightarrow (0,1]$.
\end{rem}

In order to get some intuition of what an ``typical'' function satisfying  $\Omega\in L^1(dP)$ and $\Omega\in V_\beta$ looks like, the reader is welcomed to think of a function, whose graph consists of an infinite number of ``pits'' and ``hills'', see the pictures of Section 1.5 in paper~\cite{nazhav} and Figure~\ref{kartinka} below. Of course, the same intuition applies to the functions with Lipschitz logarithm and finite logarithmic integral (i.e. those satisfying the conditions of the First Beurling--Malliavin Theorem). However, we shall shortly see that there are drastic differences between these classes of functions. 

\begin{center}
\begin{figure}
\begin{tikzpicture}
  \draw[black, ultra thick, ->] (0, -0.5) -- (0, 4.5) node[right]{$\Omega(x)$};
  \draw[black, ultra thick, ->] (-0.5, 0) -- (10, 0) node[below]{$x$};
  \fill [black] (canvas cs:x=0cm,y=4cm)   circle (2.5pt);
  \fill [black] (canvas cs:x=5cm,y=0cm)   circle (2.5pt);
    \fill [black] (canvas cs:x=9cm,y=0cm)   circle (2.5pt);
  \draw[dotted] (0, 0) grid (9, 4);
\node[left] at (0, 4)    {$\frac{2^n}{n}$};
\node[below] at (5, 0)    {$2^n$};
\node[below] at (9, 0)    {$2^{n+1}$};
  \draw[black, line width = 0.5mm, black]   plot[smooth, domain=0:4] ({5.3+0.1 * sqrt(\x)}, \x);
    \draw[black, line width = 0.5mm, black]   plot[smooth, domain=0:4] ({5.7-0.1 * sqrt(\x)}, \x);
    
     \draw[black, line width = 0.5mm, black]   plot[smooth, domain=0:4] ({6.3+0.1 * sqrt(\x)}, \x);
    \draw[black, line width = 0.5mm, black]   plot[smooth, domain=0:4] ({6.7-0.1 * sqrt(\x)}, \x);
      \draw[black, line width = 0.5mm, black]   plot[smooth, domain=0:4] ({8.3+0.1 * sqrt(\x)}, \x);
      \fill[black] (7.3, 0.1) circle (0.05);
       \fill[black] (7.5, 0.1) circle (0.05);
        \fill[black] (7.7, 0.1) circle (0.05);
    \draw[black, line width = 0.5mm, black]   plot[smooth, domain=0:4] ({8.7-0.1 * sqrt(\x)}, \x);
  % \draw[black, line width = 0.1mm, green!80!blue]   plot[smooth, domain=0:4] (2+0.01 * \x * \x, \x);
    % \draw[black, line width = 0.1mm, green!80!blue]   plot[smooth, domain=0:4] (\x * \x, \x);
  %\fill[green!80!blue] (2, 2) circle (0.2) node[below, outer sep = 5pt, black]{$(x, y)$};
  %\fill[red] (1, 4) circle (0.2) node[right, outer sep = 5pt, black]{$(1, 4)$};
 % \draw[red, dashed, thick] (1, 4) -- (2, 2) node[midway, left, outer sep = 5pt]{$d$};

  %\foreach \x in {1, 2, 3} \node[below] at (\x, 0){$\x$};
  %\foreach \y in {0, 1} \node[left] at (0, \y){$\y$};
\end{tikzpicture}
  \caption{A ``typical'' function in $V_\beta$}
  \label{kartinka}
\end{figure}
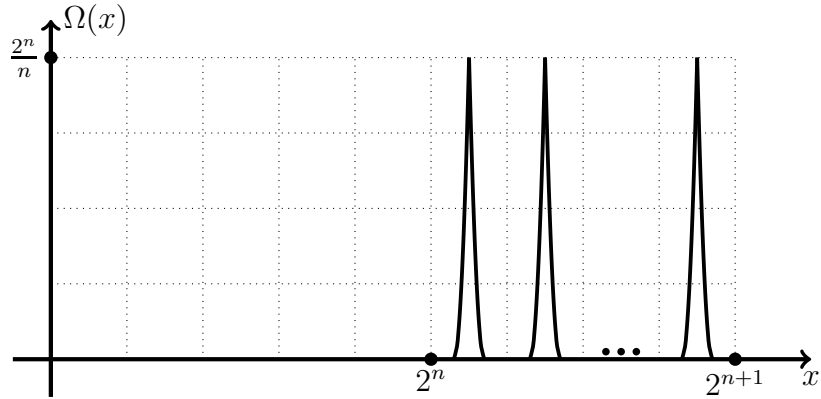
\end{center}

Indeed, let us compare our sufficient condition of Theorem~\ref{thm0} with these already known. First, it is obvious that our theorem is a generalization of the First Beurling--Malliavin Theorem, since it is a particular case of our result that corresponds to $\beta=1$. %Moreover, the new case $\beta\in (0,1)$ can not be derived from the Beurling--Malliavin Theorem in view of the following result. 

%\begin{prop}
%\label{prop4}
%For any $\beta\in (0,1)$, there are functions $\Omega:\mathbb R\rightarrow \mathbb R_+$ satisfying $\Omega\in L^1(dP)\cap V_\beta$  that can not be majorized by a function in $L^1(dP)\cap\mathrm{Lip}(\mathbb R)$.
%\end{prop}

%Indeed, one might think that Theorem~\ref{thm0} could be derived in full generality directly from Theorem~\ref{thm-1} in the following way. Take a function $\omega$ as in Theorem~\ref{thm0}. If there was another function $\varpi$ such that: $\log(1/\varpi)$ is Lipschitz, $\log(1/\varpi) \in L^1(dP)$ and $\log(1/\omega)\geq \log(1/\varpi)$, then, obviously, we would have $\omega\geq \varpi$, and, thanks to Theorem~\ref{thm-1}, $\omega$ would be $\sigma$-admissible for any $\sigma>0$. However, this argument does not apply (at least as is), in view of Proposition~\ref{prop4}.

A ``typical'' function in classes $V_\beta$ is visualized at Figure~\ref{kartinka}.

There are many other sufficient ``regularity'' conditions for the admissibility, see for instance those contained in~\cite{koos1},~\cite{koos2},~\cite{belhav}.  However, all these conditions are either imposed on the Hilbert transform of $\Omega$, or they claim that only some regularization or some minorant of $\omega$ is admissible.   
For instance, if the condition $(\log\omega(\cdot))/(1+(\cdot)^2)^{1/2}\in \dot{W}^{1/2,2}(\mathbb R)$ is fulfilled for a function $\omega$ that has convergent logarithmic integral, then some regularization of  this function is an admissible majorant, see~\cite{bermal1}. 
%then  necessarily belongs to the homogeneous fractional Sobolev--Slobodecki space $\dot{W}^{1/2,2}(\mathbb R)$.  %satisfy the so-called Dirichlet energy condition 
 %is fulfilled for some function 
%and if on top of that $\omega$ 
%condition is verified by a function $\omega$, which 
 %that has convergent logarithmic integral, then some regularization of %$\omega$ this function is an admissible majorant, see~\cite{bermal1}.

\bigskip

Note that the following approximation property of the spaces $V_\beta$ with $\beta\in(0,1]$ is a direct consequence of our Theorem~\ref{thm0}.
\begin{cor}
By a theorem of A. Baranov and V. Havin (see paper~\cite{barhav}, Section 6), we get that for any $\beta\in(0,1], \sigma>0$, and any $\omega\in V_\beta$ the space of all functions in $L^2(\mathbb R)$ with the spectrum in $\mathbb R \backslash [0,\sigma]$ is not dense in the weighted Lebesgue space $L^1(\omega)$.
\end{cor}
%We hope that our Theorem~\ref{thm0} will find other applications is harmonic analysis.

We hope that our main results will find other applications in Harmonic and Complex Analysis, in particular for the uncertainty principle and for 
%in theories of uncertainty principle and of 
 exponential systems.

%Let $\mathrm{Lip}(\mathbb R,\kappa)$ denote all Lipschitz functions in $\mathbb R$ with the Lipschitz constant $C$ satisfying $C\leq \kappa$. 

The main step of the proof of Theorem~\ref{thm0} is the following lemma.
\begin{lemma}(a new variant of the global Nazarov lemma)
\label{thm1}
Let %$1<r$ , 0<\beta$ and let 
$0<\beta\leq 1$.
%$r\alpha-2r+1+\beta r -\beta<0$ 
 Suppose that $\Omega\in L^1(dP)\cap V_{\beta}$ is positive.  
%absolutely continuous function such that $\Omega\in V_r$. 
 Then, for each $\varepsilon > 0$ there exists a function $\Omega_1$, satisfying 
\begin{enumerate}[label=(\Alph*)]
\item \label{a} $\Omega(x) \leq \Omega_1(x)$ for all $x\in \mathbb R;$
\item \label{b} $\Omega_1 \in L^1(\mathbb R, dx/(1+x^2));$
%\item $\Omega_1 \in \mathrm{Lip}(\varepsilon, \mathbb R^2);$
\item \label{c} $\mathcal H\Omega_1 \in \mathrm{Lip}(\varepsilon,\mathbb R),$ where $\mathcal H$ is the Hilbert transform on the real line.
\end{enumerate}
%\end{lemma}
\end{lemma}
Indeed, Theorem~\ref{thm0} follows from Lemma~\ref{thm1}, thanks to the following sufficient condition for a function to be a $BM$ majorant, which is a consequence of a more general result, proved by Mashreghi and Havin.
\begin{theom}
\label{diakon}
If $\omega: \mathbb R\rightarrow (0,1],\; \log(1/\omega)\in L^1(dP)$ and $\|(\mathcal H \log(1/\omega))^{\prime}\|_{\infty}< \pi\sigma,$
then $\omega$ is a $\sigma$-admissible majorant.
\end{theom}
\begin{rem}
The proof of Theorem~\ref{diakon} uses a one-dimensional construction coming from the classical (complex) theory of Hardy spaces on the unit circle. Namely, given a nonnegative function on the unit circle with convergent logarithmic integral there exists an analytic function whose modulus coincides with the former function. Such functions are called outer, see~\cite{nikol} for details.
\end{rem}
For necessary conditions for $\sigma$-admissible majorants, see papers~\cite{bel1} and~\cite{bel2} by Y. Belov and  paper~\cite{barhav} by A. Baranov and V. Havin.

We briefly discuss main ideas lying behind our proof of Lemma~\ref{thm1}. Our proof is inspired by that of the Nazarov Lemma from paper~\cite{nazhav}. Indeed, we use the beautiful idea of a so-called regularized system of intervals, which was first introduced by F. Nazarov and fruitfully used in~\cite{nazhav}. Another important feature of the proof of the Nazarov Lemma in~\cite{nazhav} is a version of the Hadamard--Landau inequality. We have had to modify this result drastically in order for it to fit the conditions of our Lemma~\ref{thm1}. This culminated in Lemma~\ref{eprst} of the present paper. On top of that, most estimates from the proof in~\cite{nazhav} become considerably harder under our assumptions, in comparison to the Lipschitz condition from~\cite{nazhav}.

Note that Nazarov's Lemma is by itself a highly nontrivial and very interesting result in Harmonic Analysis. To illustrate this, we mention paper~\cite{sz} by P. Zatitsky and D. Stolyarov, where the authors have utilized the main object of the Nazarov Lemma, the regularized system of intervals, in some special form. For a multidimensional version of the classical Nazarov Lemma, see our paper~\cite{vasil}. 

\bigskip

%TUT VSYO CHTO NIZHE NADO PEREPISAT'

Let us now discuss the second main result of this article. Our Theorem~\ref{thm1.5}, gives an answer to the following question: ``How sharp is the result of Theorem~\ref{thm0} ?'' The answer to this question is given in the following result.
% and the main reason why is, vaguely speaking, that a function in the class $V_1$ can have linear growth at sufficiently long intervals, tending to infinity. This is inconsistent with the following theorem.

\begin{theorem}
\label{thm1.5}
For any $\beta\in (0,1)$, there are functions $\omega:\mathbb R\rightarrow (0,1]$ satisfying $\log(1/\omega)\in  L^1(dP)$ and $2^{-|j|}\kappa_j^{1/(1-\beta)}\asymp 1$ in the notations of Theorem~\ref{thm0}, that are not $BM$ majorants.
%Let $\omega$ be a continuous function. If there exists a sequence of real numbers $\{x_n\}_{n\in\mathbb N}$ such that $0<x_1<x_2<\ldots, \, \lim_{n\rightarrow \infty}x_n=\infty$ and $\omega(x_n)\leq e^{-x_n}$ for all natural $n$, then $\omega$ is not a $\mathrm{BM}$ majorant.
%$\omega\equiv 0.$
\end{theorem}

We remark that %the discussion just above and 
our Theorem~\ref{thm1.5} shows that the condition $\log(1/\omega)\in V_\beta$ in our Theorem~\ref{thm0} is sharp in a number of senses.

The proof of Theorem~\ref{thm1.5} builds upon one construction from paper~\cite{belhav}. This construction says that smallness of a bandlimited function is ``contagious'': if such a function is small on an interval it is also small on a much larger concentric interval. This construction is due to A. Borichev and it works only for majorants that have a growth strictly greater than linear at a sequence tending to infinity. Majorants that appear in the formulation of Theorem~\ref{thm1.5} have at most linear growth at infinity. Nevertheless, for some of these majorants, we were able to use a combination of Borichev's construction with an iteration method to prove Theorem~\ref{thm1.5}.

\begin{comment}
The very same ``stability of smallness'' argument of Borichev yields the fact that ``typical'' functions in classes $V_r\backslash \mathrm{Lip}(\mathbb R)$ are not strongly admissible majorants, see~\cite{bel3} for a definition of the strong admissibility. This is discussed in our third remark after the proof of Theorem~\ref{thm1.5}.
\end{comment}

\bigskip
The paper is organized as follows. Theorem~\ref{thm0} is proved in Sections 2 and 3. The fourth Section is devoted to the proof of Theorem~\ref{thm1.5}. %The elementary proof of Proposition~\ref{prop4} %is given in the Appendix. %
is left to the reader and are not detailed here. %The Appendix contains examples of functions verifying conditions of our Theorem~\ref{thm0} and whose weighted logarithms have unbounded fractional homogeneous Sobolev space $\dot{W}^{1/2,2}(\mathbb R)$ seminorms.
%mentioned above are given in the Appendix, which is our Section 5.

We finally mention some open questions concerning Theorems~\ref{thm0} and~\ref{thm1.5}. The first question consists in determining whether the condition $\log(1/\omega)\in V_\beta$ in Theorem~\ref{thm0} can be weakened down to, roughly speaking, a condition of the kind ``$\omega$ belongs to some Orlicz type class, defined in the spirit of $V_\beta$ classes''. The second question concerns the system of intervals that are used in the definition of the spaces $V_\beta.$ Namely, we would like to find a necessary and sufficient condition on the system of intervals instead of the dyadic system in Definition~\ref{opredelenie}, for which the first theorem still holds.
%Indeed, still open is the question of finding a necessary and sufficient condition on the alluded system of intervals, under which we have the first theorem. 
Yet another question is to find a multidimensional version of Theorem~\ref{thm0} which seems unavailable at the present time, according to~\cite{schlhan}. The fourth and the final question reads as follows. It would be also interesting to find counterparts of the main results of this paper in the context of the so-called model spaces, in spirit of Yu. S. Belov's early papers. The author plans to attack the aforementioned questions in the nearest future.
 
\subsection*{Acknowledgments}
The author is deeply grateful to  Anton D. Baranov, Alexander A. Borichev, Evgueni S. Doubtsov, Konstantin M. Dyakonov, Sergei V. Kislya\-kov, and to Yiyu Tang for a number of helpful discussions. The author would also like to thank the anonymous referees for the careful reading of the paper and for a number of helpful suggestions.

\section{A new local Nazarov lemma}

We accumulate here the list of the frequently used technical abbreviations and notations. For an interval $a\subset \mathbb R$ its length is denoted by $l(a)$, $c_a$ will stand for the center of $a$ and $\lambda a$ with $\lambda$ positive will be the interval centered at $c_a$ and whose edge length equals $\lambda l(a)$. Let $I$ be an interval on the real line. We will denote by $T_{I}(x)$ the distance from $x\in \mathbb R$ to $\mathbb R \backslash I$. For a dyadic interval $b$, we will denote by $b^\sharp$ the dyadic parent of $b$. Throughout this paper, $I^*$ will denote the unit interval $[-1/2,1/2]$. For $\beta\in (0,1)$, we denote $\mathrm{Hol}_\beta(\kappa,I)$ the class of $\beta$-H\"older functions on the interval $I$, with the constant $\kappa$, i.e. all $f$ defined on $I$ such that for all $x\in I$ and $y\in I$ holds $|f(x)-f(y)|\leq \kappa|x-y|^\beta$.

The main step of the proof of our new global Nazarov lemma is its following local variant. 
\begin{lemma}(a new local Nazarov lemma)
\label{local}
%Let $Q^*=[-1/2,1/2]^2\subset \mathbb R^2$
%числа $\delta \in (0,1)$ и 
%and let $\kappa\geq 1$ be a real number. Suppose that $f\in \mathrm{Lip}(\kappa,Q^*)$ be a nonnegative function such that $\|f\|_{L^\infty(Q^*)}\leq 1.$ Then there exists a function $F\in C^{\infty}(\mathbb R^2),$ such that
Let $I\subset \mathbb R$ be an interval and let $\beta\in(0,1]$. Suppose that $f$ is a nonnegative absolutely continuous function such that holds $f\in \mathrm{Hol}_\beta(\kappa, I)$ %for some $0<\alpha$ 
 and $\|f\|_{L^\infty(I)}\leq \delta l(I)$  for some $0<\delta\leq 1$ and $1\leq \kappa$. Then there exists a nonnegative function $F\in \mathrm C^{\infty}(\mathbb R),$ such that
\begin{enumerate}[label=\roman*)]
\item \label{one}$F=0$ outside $1.5I,$
\item \label{two} $ f(x)\leq F(x)$ for all $x\in I,$
%\item $||F'||_{\infty}\lesssim 1;$
\item \label{three} $\| (\mathcal HF)^{\prime}\|_{L^\infty(\mathbb R)}\lesssim \delta,$ 
\item \label{four} $\int_{\mathbb R}F(x) dx \lesssim \int_I f +  \kappa \delta^{-\beta}l(I)^{1-\beta}\left(\int_I f\right)^{\beta}
%l(I)^{\frac{2}{2r-1}}\left(\int_I f\right)^{\frac{2r-2}{2r-1}}\kappa^{\frac{2r}{2r-1}}
.$
\end{enumerate}
\end{lemma}

In the case when $\beta=1$ in Lemma~\ref{local}, the corresponding result coincides with Lemma 2.6 from paper~\cite{nazhav}. %Note that there is no difference between the powers $r/(r-1)$ and $1$ in this case.

In the formulation of Lemma~\ref{local} and until the end of the third section, the signs $\lesssim$ and $\gtrsim$ indicate that the left-hand (right-hand) part of an inequality is less than the right-hand (left-hand) part multiplied by a constant independent of $\delta, f, \kappa$ and $I$. 

The rest of this section is entirely devoted to the proof of Lemma~\ref{local}. 

\begin{proof} (of the new local Nazarov lemma)

The following definition is very important.
\begin{defin}
We say that a dyadic interval $a\subset I$ is essential, if $\|f\|_{L^{\infty}(a)}\geq \delta l(a)/2.$ %where $\beta:=(r-1)/r.$ 
Denote $A$ the set of essential intervals.
\end{defin}

%Note that if $r>1,$ then $1\geq\beta>0.$ 
It is straightforward to see that we have
$$\{x\in I:f(x)>0\}\subseteq\bigcup_{a\in A}a.$$
However, we will not use this fact later on in our estimates.

Consider $A^M$, the set of maximal by inclusion elements of $A$.
To each interval $a\in A^M$ we associate its tail $t(a)$.  Informally, the tail  $t(a)$ is a family of dyadic intervals that is composed of a countable number of finite series $t_p(a),p=0,1,2,\ldots$ of dyadic intervals. For $p=0$ we define $t_0(a):=a$ and for a fixed $p\geq 1$, the intervals of the family $t_p(a)$ all have length equal to $l(a)/2^p$ and their unions form  the sets 
\begin{multline*}
	a\cup \bigcup_{1\leq q \leq p} t_q(a)\\
	=\{x\in \mathbb R: l(a)/2+l(a)\sum_{q=1}^{p-1}3^q/2^q \leq|x-c_a|< l(a)/2+l(a)\sum_{q=1}^{p}3^q/2^q \}.
\end{multline*}

For a detailed discussion of tails, see paper~\cite{nazhav}, Section 2.6.5. In fact, after that we will have added these tails, we will get a regularized system of intervals, see~\cite{nazhav}, Sections 2.6 and 2.7. Next, we define $B:=\bigcup_{a\in A^M}t(a)$, and then pose $\tau:=\{c\in B^M:c\subseteq I\}.$ Here, $B^M$ stands for the set of maximal by inclusion elements of $B$. Note that the system $\tau$ covers $I$, consists of dyadic intervals and any $c\in \tau$ satisfies $\delta l(c)\geq \|f\|_{L^\infty(c)}$, see~\cite{nazhav}. 

Define for an interval $a\in \tau$ its neighborhood $N(a)$ by
$$N(a):=\{b\in \tau: d(a,b)\leq 2l(a), \frac{1}{2}\leq \frac{l(a)}{l(b)}\leq 2\}.$$
Note that $\#N(a)\leq 9.$ We shall need the following property of the system $\tau.$
\begin{lemma}
	\label{trois}
 Suppose that $a\in \tau$ and $b\in \tau \backslash N(a).$ If $l(b)\leq 2l(a)$ then $d(2a,2b)\geq l(a)/2$, and if $l(b)=2^k l(a)$ for some natural $k\geq 2$, then $d(2a,2b)\geq 2\cdot 3^{k-2}l(a).$ 
\end{lemma}
\begin{proof} The proof of this lemma is not detailed here, since it can be found in paper~\cite{nazhav}, Section 2.6.6.
\end{proof}

Denote $2\tau:=\{2c : c \in \tau\}.$ As a direct consequence of the lemma, we deduce that the multiplicity
$\#\{b\in 2\tau:x\in b\}$ is uniformly bounded in  $x\in \mathbb R$. 
Indeed, if $b\in\tau\backslash N(a),$ then $d(2a,2b)>0$ and
$$\sup\limits_{x\in \mathbb R}\#\{b\in 2\tau:x\in b\}\leq \sup\limits_{a\in\tau}\#N(a)\lesssim 1.$$

\bigskip

Fix a bump function $\phi$, i.e. $\phi \in \mathrm C^{\infty}(\mathbb R)$ satisfying $0\leq \phi(x)\leq 1$ for all $x\in \mathbb R$, $\phi\equiv 0$ outside $1.5I^*$ and $\phi\equiv 1$ on $I^*$. Second, for an interval $a\in \tau$ pose $$\phi_a(\cdot):=\delta l(a)\phi\Bigl(\frac{(\cdot)-c_a}{l(a)}\Bigr).$$
Simple calculation shows that 
$$\mathcal H\phi_b(\cdot)=\delta l(b)\mathcal H\phi\left(\frac{(\cdot)-c_b}{l(b)}\right).$$ 
Hence we infer the inequality $\|(\mathcal H\phi_b)^{\prime}\|_{L^{\infty}(\mathbb R)}\lesssim \delta.$
We finally define $F$ by 
$$F:=\sum_{a\in \tau} \phi_a.$$

Now we have to check the required properties of the majorant $F.$ The first one follows readily from the definition of $F.$ To prove the second one, note that for all $a\in \tau$ we have $\|f\|_{L^{\infty}(a)}\leq \delta l(a)$.  Indeed, suppose the contrary, i.e. that $\|f\|_{L^{\infty}(a_0)}> \delta l(a_0)$ for some $a_0\in \tau$. This means that
$$\|f\|_{L^{\infty}(a_0^\sharp)}\geq \|f\|_{L^{\infty}(a_0)}> \delta l(a_0)=\delta \biggl(\frac{l(a_0^\sharp)}{2}\biggr),$$
which in turn signifies that $a_0^\sharp$ is an essential interval and hence $a_0^\sharp\in\tau$. This contradicts to the definition of $\tau.$ From here we deduce that if $x\in a\in \tau$, then $$F(x)\geq \delta  l(a)\geq \|f\|_{L^{\infty}(a)}\geq f(x).$$

Next we estimate the integral of the function $F$. To this end, we prove a variant of the Hadamard--Landau inequality which is appropriate for our goals.
\begin{lemma}
\label{eprst}
Let $a$ be an interval such that 
%and let $\kappa_a$ denote the following norm $\kappa_a:=(\int_a |f^\prime|^r)^{1/r}$. %and let $f$ satisfy $\|f^{\prime}\|_{L^r(a)}\leq \kappa l(a)^{1/r}$ with $\kappa$ and $r$ as above (i.e., $\kappa\geq 1$ and $r>1$). %Denote $\alpha:=(2r-2)/(2r-1)$. 
 $a\in A^M$. Then we have 
$$\|f\|^2_{L^{\infty}(a)}\lesssim \biggl(\int_a f\biggr)\delta+\kappa \biggl(\int_a f\biggr)^\beta (\delta l(a))^{1-\beta}
%\biggl(\int_a |f^\prime|^r \biggr)^{(1-\alpha/2)/r}
,$$
where $C(r)$ is a positive constant, depending on $r$ only.
\end{lemma}

\begin{proof} 
Let $x_0\in a$ be a point such that $\|f\|_{L^{\infty}(a)}=f(x_0).$ Suppose with no loss of generality that $a_+-x_0\geq l(a)/2,$ where $a_+$ is the right end of the interval $a.$ Consider a point $x\in(x_0,a_+)$. %Denote $\kappa_a:=(\int_a |f^\prime|^r)^{1/r}$. 
%\geq l(a)^{1/r}$. 
% We make use of the Newton--Leibniz formula and of the 
 Since $f$ is H\"older continuous, we hence infer the following estimate
%Suppose, as we can, that $\kappa_a:=(\int_a |f^\prime|^r)^{1/r}\geq l(a)^{1/r}$. Indeed, this assumption %$\kappa_a\geq l(a)^{1/r}$ for all $a\in \tau$ does not violate any generality, because if $\kappa_a<l(a)^{1/r}$ then it suffices to multiply the function $f$ by $(l(a)^{1/r}/\kappa_a)$ and to apply the conclusion of the lemma. %and we shall further suppose this for all such intervals.
%Let $x_0\in a$ be a point such that $\|f\|_{L^{\infty}(a)}=f(x_0).$ Suppose with no loss of generality that $a_+-x_0\geq l(a)/2,$ where $a_+$ is the right end of the interval $a.$ Consider a point $x\in(x_0,a_+)$. We make use of the Newton--Leibniz formula and of the H\"older inequality to infer the following estimates
\begin{equation*}
%\begin{split}
f(x)\geq f(x_0)-\kappa (x-x_0)^{\beta}.
%|&=\Bigl|\int_{x_0}^x f'(t) dt\Bigr| \\
%&\leq |x-x_0|^{1-1/r}\left(\int_{x_0}^x |f'(t)|^r %dt\right)^{1/r} \leq \kappa_a |x-x_0|^{1-1/r}.
%\end{split}
\end{equation*}
%thanks to the Newton--Leibniz formula and to the H\"older inequality. 

Denote $\upsilon:=(f(x_0)/\kappa)^{1/\beta}$. We shall treat two cases separately, according to the value of $\upsilon$. First, we suppose that $\upsilon<l(a)/2$. Observe that in this case the point $x_0+\upsilon$ belongs to the interval $a$. We integrate the estimate just above  %$f(x)\geq f(x_0)-\kappa_a |x-x_0|^{1-1/r}$ 
using this observation and deduce that
%Observe that the point $x_0+\upsilon/2$ belongs to the interval $a$. This follows from the inequality $\|f\|_{L^\infty(a)}\leq l(a)$, which is true since $a\in\tau$. We integrate the inequality $f(x)\geq f(x_0)-\kappa_a |x-x_0|^{1-1/r}$ using this observation and deduce that
\begin{equation}
\label{nadoelo}
\begin{split}
\int_a f &\geq \int_{x_0}^{x_0+\upsilon/2} f(x) dx \geq  \int_{x_0}^{x_0+\upsilon/2} f(x_0)-\kappa (x-x_0)^{\beta} dx\\
&= \frac{f(x_0)}{2}\left(\frac{f(x_0)}{\kappa}\right)^{1/\beta} - \frac{\kappa}{2^{1+\beta}(\beta+1)}\left(\frac{f(x_0)}{\kappa}\right)^{(\beta+1)/\beta}\\
&\gtrsim\frac{\|f\|_{L^\infty(a)}^{(\beta+1)/\beta}}{\kappa^{1/\beta}}=\frac{\|f\|_{L^\infty(a)}^{2/\beta}\|f\|_{L^\infty(a)}^{(\beta-1)/\beta}}{\kappa^{1/\beta}}\gtrsim\frac{\|f\|_{L^\infty(a)}^{2/\beta}(\delta l(a))^{(\beta-1)/\beta}}{\kappa^{1/\beta}},
\end{split}
\end{equation}
%But since $a\in A^M$, we have also $\kappa_a\leq \kappa l(a)^{1/r}\leq \kappa (2\|f\|_{L^\infty(a)})^{1/r}.$ This and lines~\eqref{nadoelo} give the claimed inequality in the first case.
where the last bound above follows from the fact that $a\in A^M$. Hence we have that
$$
\|f\|_{L^\infty(a)}^2\lesssim \kappa \biggl(\int_a f\biggr)^\beta (\delta l(a))^{1-\beta}.
$$

\bigskip

Consider now the second case, where $\upsilon\geq l(a)/2$. In this case we shall use the fact that the point $x_0+l(a)/2$ belongs to the interval $a$. Integrating the same inequality as in the first case yields
\begin{equation}
\label{nadoelo2}
\begin{split}
\int_a f &\geq \int_{x_0}^{x_0+l(a)/2} f(x) dx \\
&\geq  \frac{l(a)f(x_0)}{2}-\frac{\kappa l(a)^{\beta+1}}{2^{\beta+1}(\beta+1)}=\frac{l(a)}{2}\left(f(x_0)-\frac{\kappa}{\beta+1}\cdot\left(\frac{l(a)}{2}\right)^{\beta}\right).
\end{split}
\end{equation}
Note that since $\upsilon\geq l(a)/2$, we have also that $f(x_0)/\kappa_a\geq (l(a)/2)^{\beta}$. Let us use this in the following way
\begin{equation*}
\int_a f \geq \frac{l(a)}{2}\left(f(x_0)-\frac{f(x_0)}{\beta+1}\right)\gtrsim \delta^{-1} \|f\|^2_{L^\infty(a)},
\end{equation*}
 %This and inequality~\eqref{nadoelo2} prove the claim in the second case.
thanks to the fact that $a\in A^M$. Hence Lemma~\ref{eprst} is proved.
\end{proof}
%\begin{rem}
%The assumption $\kappa_a>l(a)^{1/r}$ for all $a\in \tau$ does not violate any generality and we shall further suppose this for all such intervals.
%\end{rem}

So, let us start the estimates of the integral of the function $F$:

\begin{equation}
\label{chastnoe}
\begin{split}
&\int_{\mathbb R} F \leq \sum_{b\in A^{M}} \int_{\mathbb R} \phi_b +  \sum_{c\in A^{M}} \sum_{b\in t(c)\backslash c} \int_{\mathbb R} \phi_b \leq \delta\sum_{b\in A^{M}}l(b)^{2} + \delta\sum_{c\in A^{M}}\sum_{b\in t(c)\backslash c} l(b)^{2} \\ 
 &\lesssim\delta\sum_{b\in A^{M}}l(b)^{2} + \delta\sum_{c\in A^{M}}\sum_{p=1}^{\infty}\sum_{b\in t_p(c)} l(b)^{2} \lesssim \delta\sum_{b\in A^{M}} l(b)^{2} + \delta\sum_{c\in A^{M}}\sum_{p=1}^{\infty} 3^p\biggl(\frac{l(c)}{2^p}\biggr)^{2} \\
&\lesssim\delta\sum_{c\in A^{M}} l(c)^{2} \lesssim \delta^{-1}\sum_{c\in A^{M}}\|f\|_{L^\infty(c)}^{2}.
\end{split}
\end{equation}

We further use the result of Lemma~\ref{eprst} to continue the estimates of the integral of the function $F$:

\begin{multline}
\label{chastnoe}
\int_{\mathbb R} F\lesssim \sum_{c\in A^{M}}  \int_c f+\delta^{-1}\sum_{c\in A^{M}}\kappa \biggl(\int_c f\biggr)^\beta (\delta l(c))^{1-\beta}\\
 \leq \int_I f +  \delta^{-\beta}\kappa\biggl(\sum_{c\in A^{M}}\int_c f\biggr)^{\beta}\cdot \biggl(\sum_{c\in A^{M}}l(c)\biggr)^{1-\beta}\\
\lesssim \int_I f +  \delta^{-\beta}\kappa l(I)^{1-\beta}\biggl(\int_I f\biggr)^{\beta}.
%\sum_{c\in A^{M}}\biggl(\int_c f\biggr)\kappa^{r/(r-1)}\leq \kappa^{r/(r-1)}\int_I f.
%\\ \leq\biggl(\sum_{c\in A^{M}}\int_c f\biggr)^{\alpha}\biggl(\sum_{c\in A^{M}}\kappa_c^{(2-\alpha)/(1-\alpha)}\biggr)^{1-\alpha}\leq \kappa^{2r/(2r-1)} \biggl(\int_{\mathbb R}f\biggr)^{\alpha}.
\end{multline}
The last and the last but one inequalities just above are in need of explanation. The last estimate uses the fact that intervals of $A^M$ are nonoverlapping, whereas the penultimate bound follows from the H\"older inequality.
%$\kappa_c<\kappa$ and that $2-\alpha=2r/(2r-1)$, whereas the penultimate bound follows from the H\"older inequality with exponents $1/\alpha>1$ and its conjugate $1/(1-\alpha).$

It remains to derive the inequality on the derivative of the Hilbert transformation of the function $F$. First, we shall obtain this estimate for $x\in \bigcup_{b\in\tau}{2b}.$ Let $a(=a(x))$ denote the interval from $\tau$ such that $x\in 2a$. We isolate the neighborhood $N(a)$ from its complement in $\tau$ and infer the following inequality
\begin{multline*}
|(\mathcal HF)^\prime(x)|\\
\leq \sum_{b\in N(a)}|(\mathcal H\phi_b)^\prime(x)| + \sum_{\substack{b\in \tau\backslash N(a),\\ l(b)\leq 2l(a)}}|(\mathcal H\phi_b)^\prime(x)|+ \sum_{k=2}^{\infty}\sum_{\substack{b\in \tau\backslash N(a),\\ l(b)=2^k l(a)}}|(\mathcal H\phi_b)^\prime(x)|=: S_1+S_2+S_3.
\end{multline*}

We shall estimate the terms $S_1, S_2$ and $S_3$ separately. We start with the sum $S_1,$ whose estimate turns out to be easy
$$S_1\leq \#N(a) \sup_{b\in \tau}\|(\mathcal H\phi_b)^\prime\|_{L^{\infty}(\mathbb R)}\lesssim \delta.$$

We further proceed to the second term. We use a simple estimate on the kernel of the Hilbert transformation, the fact that the system of intervals $\{2b\}_{b\in\tau}$ (by Lemma~\ref{trois}) has finite multiplicity and Lemma~\ref{trois} to get
\begin{multline*}
S_2\lesssim \sum_{\substack{b\in \tau\backslash N(a),\\ l(b)\leq 2l(a)}}\int_{\mathbb R}\phi_b(t)\frac{\partial}{\partial x}\Bigl(\frac{1}{t-x}\Bigr)dt \lesssim\sum_{\substack{b\in \tau\backslash N(a),\\ l(b)\leq 2l(a)}}\int_{\mathbb R}\frac{\phi_b(t)}{(t-x)^2}dt \\
 \lesssim\sum_{\substack{b\in \tau\backslash N(a),\\ l(b)\leq 2l(a)}}\int_{1.5b} \frac{\delta l(a) dt}{(t-x)^2}\lesssim \delta l(a)\int_{\{|u|\geq l(a)/2\}}\frac{du}{|u|^2}\lesssim \delta.
\end{multline*}

The third term  can be estimated as well using Lemma~\ref{trois}:
\begin{multline*}
S_3\lesssim \sum_{k=2}^{\infty}\sum_{\substack{b\in \tau\backslash N(a),\\ l(b)=2^k l(a)}} \int_{2b}\frac{\phi_b(t)dt}{(t-x)^2}\leq \sum_{k=2}^{\infty}2^k \delta l(a)\sum_{\substack{b\in \tau\backslash N(a),\\ l(b)=2^k l(a)}} \int_{2b}\frac{dt}{(t-x)^2}  \\ \lesssim\sum_{k=2}^{\infty}2^k\delta l(a) \int_{\{|u|\geq 2\cdot 3^{k-2}l(a)\}}\frac{du}{u^2}\lesssim \delta,
\end{multline*}
and the lemma for $x\in \bigcup_{b\in\tau}{2b}$ follows. 

Next, if a point $z\in \mathbb R$ is situated at a positive distance from the set $\bigcup_{b\in\tau}{2b},$ then denote by $x$ the point of this set closest to $z$, and let $a(=a(x))$ be an interval as above. We infer the following estimates
\begin{multline*}
|(\mathcal HF)^\prime(z)|\leq \sum_{b\in \tau\backslash N(a)} |(\mathcal H\phi_b)^\prime(z)| + \sum_{b\in N(a)} |(\mathcal H\phi_b)^\prime(z)|\\
\lesssim\sum_{b\in \tau\backslash N(a)} \int_{\mathbb R}\frac{\phi_b(t)dt}{(t-x)^2}+\#N(a) \sup_{b\in \tau}\|(\mathcal H\phi_b)^\prime\|_{L^{\infty}(\mathbb R)}.
\end{multline*}
Thanks to the estimates of the terms $S_1$, $S_2$ and $S_3$, we conclude that the needed variant of the local Nazarov lemma is proved.
\end{proof}

\section{Proof of a new global Nazarov lemma}

In this section, we shall derive the global Nazarov lemma from the local one. 

\begin{proof} Until the end of the third section, the signs $\lesssim$ and $\gtrsim$ indicate that the left-hand (right-hand) part of an inequality is less than the right-hand (left-hand) part multiplied by a ``harmless'' positive constant.

%Denote $\mu:=\|\Omega\|_{V_r}$. 
Note that we may assume in the global Nazarov lemma that $\Omega(x)=0$ for $|x|\leq R$ with $R$ being an arbitrary large positive number. Indeed, if it is not the case, then consider the function $\mathit{\Omega}(\cdot)=\max(0,\Omega-\mathcal M)(\cdot),$ where $\mathcal M:=\max_{x\in B(0,R)} \Omega(x)$. If $\mathit{\Omega_1}$ is a majorant of the function $\mathit{\Omega},$ satisfying properties~\ref{b} and~\ref{c} then the function $\mathit{\Omega_1}+\mathcal M$ will be the desired majorant of the function $\Omega.$ 
%We also readily deduce that $\Omega(x)\leq \mu|x|.$

%Let us now prove the global Nazarov lemma. Denote $\mu:=\|\Omega\|_{V_r}.$  First, given $\varepsilon>0,$ we find $R$ large enough to guarantee $\int_{R}^\infty\Omega(t)dP(t)\leq \varepsilon^{2/\alpha}.$ Note that we may assume in the lemma that $\Omega(x)=0$ for $|x|\leq \max(R,\Omega(0)/\mu,1)=:\sigma.$ Indeed, if it is not the case consider the function $\widetilde{\Omega}(\cdot)=\max(0,\Omega-M)(\cdot),$ where $M:=\max_{x\in B(0,\sigma)} \Omega(x)$. If $\widetilde{\Omega_1}$ is a majorant of $\widetilde{\Omega},$ then $\Omega_1:=\widetilde{\Omega_1}+M$ will be a majorant of $\Omega.$

Fix $0<\varepsilon\leq 1$ and choose $1<R_1$ so big that 
$$\int_{\mathbb R\backslash (-R_1,R_1)}\Omega dP\leq \varepsilon.$$ 
Since the series~\eqref{ryad} converges, there hence exists a natural $N_1$ so big that for all $j>N_1$ holds $\kappa_j^{1/(1-\beta)}2^{-j}\leq \varepsilon^{1/(1-\beta)}.$ As a consequence, we infer for all such $j$ the following bound 
\begin{equation}
\label{kappppa}
\kappa_j 2^{j(\beta+1)}\leq \varepsilon 2^{2j}.
\end{equation}
Thanks to the previous paragraph, we may assume that $\Omega$ equals to zero at the interval $(-\max(R_1,2^{N_1}),\max(R_1,2^{N_1}))$. 

Recall the above defined system of intervals: $J_0=[-2,2),$ and for $j\in \mathbb N$, 
$$J_{j}=[2^{j},2^{j+1}),J_{-j}=[-2^{j+1},-2^{j}).$$ 
Next, we shall prove for $x\in \mathbb R$ the following inequality 
\begin{equation}
\label{ocbesk}
 \Omega(x)\lesssim \varepsilon|x|.
%^{\alpha}.
\end{equation}
With no loss of generality, we suppose that $x>0$ and we let $n\in \mathbb N$ be such that $2^n\leq |x|< 2^{n+1}$. First, note that according to the previous paragraph, the bound~\eqref{ocbesk} is obvious once $|x|\lesssim 1$. Second, for $1 \lesssim |x|$ we will argue as in Lemma~\ref{eprst}. We thus find a point $x_0\in [2^n,2^{n+1})$ such that $\Omega(x_0)=\|\Omega\|_{L^{\infty}(J_n)}$. Then, for any $y\in J_n$ we have
$$
\Omega(y)\geq \Omega(x_0)-\kappa_n|y-x_0|^{\beta}.
$$

%\begin{comment}
%\begin{equation}
%	\label{russkayamisl}
% \Omega(x)\leq |\Omega(x)-\Omega(2^n)|+\sum_{j=1}^{n}|\Omega(2^{j})-\Omega(2^{j-1})|
% \lesssim \sum_{j=1}^{n+1} 2^{j\beta} \kappa_j\lesssim 2^n \leq |x|,
%^{\alpha}.
%\end{equation}
%which means that $\Omega(x)\leq \mu |x|$ for some $\mu>0$. Second, according to the previous paragraph the bound~\eqref{ocbesk} is %obvious once $|x|\lesssim 1$, and for $1 \lesssim |x|$ we will argue as in Lemma~\ref{eprst}. We thus find a point $x_0\in [2^n,2^{n+1}%)$ such that $\Omega(x_0)=\|\Omega\|_{L^{\infty}(J_n)}$. Then, for any $y\in J_n$ we have
%$$
%\Omega(y)\geq \Omega(x_0)-\kappa_n|y-x_0|^{\beta}.
%$$
%for some constant $C(\mu)$ depending on $\mu$ only. %
%\end{comment}
%
%Further, denote $\epsilon:=(\Omega(x_0)/\kappa_n)^{1/\beta}.$ Due to the bound~\eqref{russkayamisl}, one of the points $x_0\pm\epsilon$ belongs to the interval $I_n$. 

Once again, without loss of generality we suppose that the point $x_0+2^n/2$  belongs to the interval $J_n$. We finally infer the following chain of inequalities

\begin{multline}
%\label{rostnabesk}
\varepsilon \geq \int \Omega dP\geq |x|^{-2}\int_{J_n} \Omega(y) dy \\
\geq |x|^{-2} \int_{x_0}^{x_0+2^n/2} \left(\Omega(x_0)-\kappa_n(y-x_0)^{\beta}\right)dy \\
\geq |x|^{-2}\left(2^{n-1}\Omega(x_0)-C(\beta)\kappa_n2^{n(\beta+1)}\right).
\end{multline}
%for some constant $C(\mu,r)$ depending on $\mu$ and $r$ only. 
Hence, the bound~\eqref{ocbesk} is proved, by virtue of~\eqref{kappppa}.

\bigskip

Apply the local Lemma to each interval $J_j$ and the corresponding restriction $f_j=\Omega\restr{J_j}$. Indeed, Lemma~\ref{local} can be applied since these functions satisfy 
$$\|f_j\|_{\infty}\leq \varepsilon 2^{j}\leq \varepsilon l(J_j)$$ 
by~\eqref{ocbesk}. % and 
%$$\|f_j^{\prime}\|_{L^r(J_j)}\leq \mu l(J_j)^{1/r}\leq\max(\mu,1)l(J_j)^{1/r},$$ 
%thanks to the fact that $\Omega$ is in $V_r$.
%and $\|f_j^{\prime}\|_r\lesssim \kappa$) 
 Thus we obtain functions $F_{j}$ for $j\in \mathbb Z.$ The needed majorant $\Omega_1$ is defined by
$$\Omega_1=\sum_{j\in\mathbb Z} F_{j}.$$

Now, we shall check the required properties of $\Omega_1.$ The first property follows obviously from the local Lemma. We proceed to the second one:
\begin{multline}
\label{mammamia}
	%\begin{split}
\int_{\mathbb R} \Omega_1(t) dP(t) = \sum_{j \in \mathbb Z} \int_{\mathbb R} F_{j}(t) dP(t) \lesssim \sum_{j\in \mathbb Z}\int_{1.5 J_{j}} F_{j}(t) \frac{dt}{2^{2|j|}} \\ 
\lesssim \varepsilon^{-\beta}\sum_{j\in \mathbb Z}2^{|j|(\beta-1)} \kappa_j \biggl(\int_{1.5 J_{j}} \Omega(t)\frac{dt}{2^{2|j|}}\biggr)^{\beta}+\sum_{j\in \mathbb Z}\int_{1.5 J_{j}}\Omega(t)\frac{dt}{2^{2|j|}}\\
 \lesssim \varepsilon^{-\beta}\biggl(\sum_{j\in \mathbb Z} 2^{-|j|}\kappa_j^{1/(1-\beta)}\biggr)^{1-\beta} \cdot\biggl(\sum_{j\in \mathbb Z}\int_{1.5 J_{j}} \Omega(t)\frac{dt}{2^{2|j|}}\biggr)^{\beta}+\varepsilon \lesssim\varepsilon,
%\end{split}
\end{multline}
where in the third inequality above we have used the local Lemma and in the penultimate bound we have used the H\"older inequality.
%$\alpha\in (0,1).$ 

So, it remains to check that the third conclusion holds. First, fix a point $x\in \mathbb R.$ Second, denote by $S(x)$ the interval from the system $\mathcal F=\{J_{j}\}_{j\in \mathbb Z}$ such that $x\in S(x).$ Next, denote by $U(x)$ the subset of $\mathcal F$ consisting of $S(x)$ and its two neighbor intervals and by $W(x)$ its complement: $W(x)= \mathcal F \backslash U(x).$ Finally, write the function $\Omega_1$ as a sum of two functions as follows:
$$\Omega_1=\sum_{j\in W(x)} F_{j}+ \sum_{j\in U(x)} F_{j}=:\omega_1+\omega_2.$$
Since there is only a finite number of intervals in the family $U(x)$, we see that
$$|(\mathcal H\omega_2)^{\prime}(x)|\lesssim \#U(x) \sup_{j\in U(x)}\|(\mathcal HF_{j})^{\prime}\|_{\infty}\lesssim \varepsilon,$$
where we have just used condition~\eqref{three} of Lemma~\ref{local} in the last estimate. On the other hand, since $\mathrm{supp} (\omega_1)\subseteq\bigcup_{j\in W(x)}1.5J_{j}$ we deduce that $$\mathrm{supp} (\omega_1)\subseteq\{t\in \mathbb R: |t-x|\geq \frac{l(S(x))}{4}\}\subseteq\{t\in \mathbb R: |t-x|\geq \frac{|x|}{16}\}.$$ Therefore, we arrive at the following chain of inequalities
\begin{multline*}
|(\mathcal H\omega_1)^{\prime}(x)|=\Bigl|\Bigl(\int_{\mathbb R}\omega_1(t)\frac{1}{t-x}dt\Bigr)^{\prime}\Bigr|=\Bigl|\int_{\mathbb R}\omega_1(t)\frac{\partial}{\partial x}\Bigl(\frac{1}{t-x}\Bigr)dt\Bigl|
\\ =\int_{\mathbb R}\omega_1(t)\frac{1}{(t-x)^2}dt \lesssim \int_{\mathbb R}\Omega_1(t) dP(t)\lesssim \varepsilon.
%\lesssim  \frac{1}{\varepsilon^{r/(r-1)}}\int_{\mathbb R}\Omega(t) dP(t) \lesssim \varepsilon<\infty,
\end{multline*} 
thanks to the bound~\eqref{mammamia}. 

Hence, the needed variant of the Nazarov lemma is  proved.
\end{proof}

Thus, Theorem~\ref{thm0} is also proved, via Theorem~\ref{diakon}.

\section{Sharpness of Theorem~\ref{thm1}}  

Note that the proof of Theorem~\ref{thm1.5}, is a direct consequence of the following proposition.

\begin{prop}
Let $\gamma>1/2$ and denote $I_n:=[2^n-2^n/n^\gamma,2^n+2^n/n^\gamma]$, for $n\geq 3$. Consider for $x\in \mathbb R$ the following function
\begin{equation}
	\label{moyastarayatsipella}
	\omega(x):=
	\begin{cases}
		\exp(-n^{\gamma-1/2} T_{I_n}(x)), \text{ if } x\in I_n \text{ with } n\geq 3,\\
		1, \text{ otherwise.}
	\end{cases}
\end{equation}
We claim that $\log(1/\omega)\in L^1(dP)$ and that $\log(1/\omega)$ satisfies the regularity assumption of Theorem~\ref{thm1.5}, though $\omega$ is not a $BM$ majorant.
\end{prop}

The graph of the function $\Omega=\log(1/\omega)$ and the main idea of the proof below (i.e. the iteration) is illustrated at Figure~\ref{kartinka2}.

\begin{center}
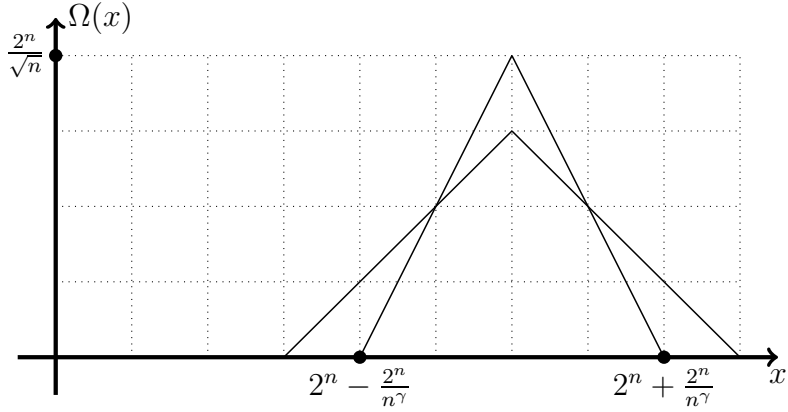
\begin{figure}
\begin{tikzpicture}
  \draw[black, ultra thick, ->] (0, -0.5) -- (0, 4.5) node[right]{$\Omega(x)$};
  \draw[black, ultra thick, ->] (-0.5, 0) -- (9.5, 0) node[below]{$x$};
  \draw[dotted] (0, 0) grid (9, 4);
  \node[left] at (0, 4)    {$\frac{2^n}{\sqrt{n}}$};
   \fill [black] (canvas cs:x=0cm,y=4cm)   circle (2.5pt);
    \fill [black] (canvas cs:x=4cm,y=0cm)   circle (2.5pt);
     \fill [black] (canvas cs:x=8cm,y=0cm)   circle (2.5pt);
\draw[black, line width = 0.2mm, black]   plot[smooth, domain=0:4] ({4+ 0.5*\x}, \x);
\draw[black, line width = 0.2mm, black]   plot[smooth, domain=0:4] ({8- 0.5*\x}, \x);
\draw[black, line width = 0.2mm, black]   plot[smooth, domain=0:3] ({3+ \x}, \x);
\draw[black, line width = 0.2mm, black]   plot[smooth, domain=0:3] ({9- \x}, \x);
\node[below] at (4, 0)    {$2^n-\frac{2^n}{n^\gamma}$};
\node[below] at (8, 0)    {$2^n+\frac{2^n}{n^\gamma}$};
  %\draw[black, line width = 0.5mm, black]   plot[smooth, domain=0:4] ({3.5+1 * sqrt(\x +1)}, \x);
  %  \draw[black, line width = 0.5mm, black]   plot[smooth, domain=0:4] ({8-1 * sqrt(\x+1)}, \x);
  % \draw[black, line width = 0.1mm, green!80!blue]   plot[smooth, domain=0:4] (2+0.01 * \x * \x, \x);
    % \draw[black, line width = 0.1mm, green!80!blue]   plot[smooth, domain=0:4] (\x * \x, \x);
  %\fill[green!80!blue] (2, 2) circle (0.2) node[below, outer sep = 5pt, black]{$(x, y)$};
  %\fill[red] (1, 4) circle (0.2) node[right, outer sep = 5pt, black]{$(1, 4)$};
 % \draw[red, dashed, thick] (1, 4) -- (2, 2) node[midway, left, outer sep = 5pt]{$d$};

  %\foreach \x in {1, 2, 3} \node[below] at (\x, 0){$\x$};
  %\foreach \y in {0, 1} \node[left] at (0, \y){$\y$};
\end{tikzpicture}
  \caption{The main idea of the proof of Theorem~\ref{thm1.5}}
  \label{kartinka2}
\end{figure}
\end{center}

\begin{proof}
	 The first two claims are easy to verify, so we omit their proofs.

Let $\sigma$ be a positive constant and consider the  Bernstein space $\mathcal{E}_{\sigma,1},$ i.e. the space of all entire functions $f$ such that
$$|f(z)|\leq e^{\sigma |z|} \text{ for any } z\in \mathbb C \text{ and }  |f|\leq 1 \text{ on } \mathbb R.$$ 

Recall Lemma 1 from paper~\cite{belhav}.
\begin{lam}
\label{klam}
 For any $\sigma>0$ there exist a (small) $\alpha(\sigma)\in (0,1/2)$ and a (big) $h(\sigma)>2$ such that for any $h\geq h(\sigma)$, any $f\in \mathcal{E}_{\sigma,1}$ and any compact interval $I\subset \mathbb R$
 $$
 |f|\leq e^{-hT_I} \text{ on } \mathbb R \implies |f|\leq e^{-Ch|I|}  \text{ on } \widetilde{I},$$
where $C > 0$ is an absolute constant and $\widetilde{I}$ is the interval centered at $c(I)$ with $|\widetilde{I}|=
%\sqrt{h^{2\alpha(\sigma)}-1}
h^{\alpha(\sigma)}|I|$.
%
%Note that $$\frac{1}{2}h^{\alpha(\sigma)}|I|\leq |\widetilde{I}| \leq h^{\alpha(\sigma)}|I|,
%$$if $h\geq h(\sigma)$ (for big values of $h(\sigma)$).
\end{lam}

Let us prove that  $\omega$ is not in the $BM$ class. Suppose the contrary. Hence, for a fixed $\sigma>0$ there exists a function, not identically zero, satisfying $f\in L^2(\mathbb R)$, $\mathrm{spec}(f)\subset [0,\sigma]$ and $|f(x)|\leq \omega(x)$ for all real $x$. We shall now use Lemma~\ref{klam}.
%Lemma 1 from paper~\cite{belhav}
  Notice that the function $f$ satisfies the conditions of this lemma with $h:=n^\vartheta$, where $\vartheta:=\gamma-1/2>0$ and $I:=I_n$ for $n\geq n(\gamma).$ We deduce from this lemma that there exists a universal constant $C$ and a power $\alpha\in (0,1/2)$, depending only on $\sigma$ such that 
  $$|f(x)|\leq \mathrm{exp}(-C2^n)$$ 
  on the interval $I_{n,1}:=(n^{\vartheta\alpha}/2)I_n$. 

\begin{rem}
 From now on until the end of the present article, the sign $X\asymp Y$ means that $\mathcal{C}_1Y\leq X\leq \mathcal{C}_2 Y$ for some constants $\mathcal{C}_1$ and $\mathcal{C}_2$ depending only on $\gamma, \alpha, \sigma$ and $C$. In this case, we shall say that $X$ is of order $Y$.
\end{rem}

Note that the length of this interval satisfies the bound $|I_{n,1}|\asymp n^{\vartheta(\alpha-1)}2^n.$ As a consequence, we infer that the inequality
$$|f(x)|\leq \exp(-Cn^{\vartheta(\alpha-1)}T_{I_{n,1}}(x))$$
is valid for $x\in I_{n,1}$. This means that we can apply Lemma~\ref{klam} once again, now for $h:=Cn^{\vartheta(1-\alpha)}$ and $I:=I_{n,1}.$ This yields the bound 
$$|f(x)|\leq e^{-Ch|I_{n,1}|}=e^{-C^22^n},$$
which is true for $x\in I_{n,2}:=((Cn^{\vartheta(1-\alpha)})^{\alpha}/2)I_{n,1}.$ It is not difficult to see that the corresponding interval $I_{n,2}$ has length of order 
$$C^\alpha n^{\vartheta(1-\alpha)\alpha+\vartheta(\alpha-1)}2^n=C^\alpha n^{-\vartheta(1-\alpha)^2}2^n.$$

Acting inductively, after $m$ steps, we arrive at the estimate $|f(x)|\leq \mathrm{exp}(-C^m2^n)$, verified by $f$ for $x\in I_{n,m}$ with $$|I_{n,m}|\asymp n^{-\vartheta(1-\alpha)^m}2^n.$$ Maybe, it is worth noting that  $I_{k,m}\cap I_{n,m}= \emptyset$ for any natural $m$, once $n\neq k.$ This results from the fact that we assume, as we can, that $C<1.$

We are now in position to prove that $f=0$ identically, which will lead to a contradiction. To this end, we estimate the logarithmic integral of $f$. For each natural number $m$ holds
$$
\int_{\mathbb R} \log|f(x)| dP(x) \leq -\sum_{n\geq 3} \int_{I_{n,m}} 2^{-2n} C^m 2^n dx \asymp -\sum_{n\geq 3} n^{-\vartheta (1-\alpha)^m}.
$$ 
Choosing $m$ sufficiently large and recalling that $(1-\alpha)\in (0,1)$, we arrive at the formula 
$$\int_{\mathbb R} \log|f(x)| dP(x)=-\infty.$$ Since $f\in L^2(\mathbb R)$ has the spectrum in the interval $[0,\sigma],$ it hence belongs to the Hardy class $H^2(\mathbb R).$ From the Jensen inequality, see~\cite{havjor} we deduce that $f=0$ identically, which contradicts our assumption. Hence, the second theorem is proved.
\end{proof}

\begin{rem} Alas, our proof above does not work, if one replaces in~\eqref{moyastarayatsipella} and in the definition of intervals  $I_n$ %in the argument of the exponential 
 the powers $n^\gamma$ by $\theta^{n}$ with $\theta\in (1,2)$.
\end{rem}

\begin{comment}
\begin{rem}
Consider for $n\in \mathbb N$ intervals $ I_{n,\ast}:=[2^{2^n},2^{2^{n}}\cdot(1+n^{-2})]$. Define for $x\in \mathbb R$ the following function
$$
\omega_\ast(x):=
\begin{cases}
e^{-n^2 T_{I_{n,\ast}}(x)}, \text{ if } x\in I_{n,\ast} \text{ with } n\geq 3,\\
1, \text{ otherwise.}
\end{cases}
$$
 The proof just above can be used to show that this function is not a $BM$ majorant. On the other hand, it is easy to see that $\omega_\ast\in L^1(dP)$ and that it satisfies
 $$
 \sup_{n\in \mathbb N} \frac{1}{(2^{2^{n+1}}-2^{2^{n}})}\int_{2^{2^n}}^{2^{2^{n+1}}}|\omega_\ast^\prime(x)|^rdx < \infty.
 $$
  This means that one can not replace $2^n$ in the definition of spaces $V_r$ with $2^{2^n}$, and hence also with a general Hadamard lacunary sequence.
\end{rem}
\end{comment}

\begin{rem}
It can be seen exactly as above that the function $\omega_\ast$ is not a strictly admissible majorant, recall Definition~\ref{horoshayamajoranta}.  For a detailed discussion of strictly admissible majorants, see~\cite{bel3}.
\end{rem}

\renewcommand{\refname}{References}

\end{document}